\documentclass{amsart}

\usepackage[utf8]{inputenc} 

\usepackage{amsmath, amssymb, amsthm, amsfonts}
\usepackage[american]{babel}
\usepackage{hyperref}
\usepackage{color}
\usepackage{tikz-cd}
\usepackage{mathtools}
\usepackage{comment}
\usepackage{MnSymbol}
\usepackage{csquotes}
\usetikzlibrary{arrows}

\newtheorem{theorem}{Theorem}
\newtheorem{lemma}[theorem]{Lemma}
\newtheorem{definition}[theorem]{Definition}

\newtheorem{corollary}[theorem]{Corollary}

\numberwithin{equation}{section}

\newcommand{\Z}{\mathbb{Z}}

\newcommand{\R}{\mathbb{R}}
\newcommand{\C}{\mathbb{C}}

\newcommand{\Sc}{\mathcal{S}}

\newcommand{\I}{\mathfrak{I}}

\newcommand{\tr}{\operatorname{tr}}
\newcommand{\supp}{\operatorname{supp}}
\newcommand{\real}{\text{Re}}

\title{A priori estimates for a quadratic dNLS}
\author{Friedrich Klaus}

\begin{document}

\maketitle

\begin{abstract}
    In this work we consider integrable PDE with higher dimensional Lax pairs. Our main example is a quadratic dNLS equation with a $3 \times 3$ Lax pair. For this equation we show a-priori estimates in Sobolev spaces of negative regularity $H^s(\R), s > -\frac12$. We also prove that for general $N \times N$ Lax operators $L$, the transmission coefficient coincides with the $2$-renormalized perturbation determinant. 
\end{abstract}

\section{Introduction}

In this work we consider integrable PDE with higher dimensional Lax pairs. Our main example is the quadratic dNLS, given by
\begin{equation}\label{eq:qdNLS}
    iq_t + \frac1{\sqrt{3}}q_{xx} + 2i\bar q \bar q_x = 0.
\end{equation}
which admits a Lax pair with a $3\times 3$ matrix Lax operator, see \eqref{eq:L}. Here, $q(t,x): I \times \R \to \C$ is a complex-valued function defined on the real line. We are interested in the initial-value problem for \eqref{eq:qdNLS} and impose the initial condition
\[
    q(0,\cdot) = q_0 \in H^s(\R).
\]
Equation \eqref{eq:qdNLS} was brought to the author's attention when reading the very interesting work \cite{charlierlenells} by Charlier--Lenells on a Miura map relating \eqref{eq:qdNLS} to the good Boussinesq equation. While he is not aware of any physical motivation to study \eqref{eq:qdNLS}, and his own motivation is purely mathematical, he is quite sure that it can be used as well to study wave propagation with a quadratic interaction.

From the scaling invariance
\[
    q_{\lambda}(t,x) = \lambda q(\lambda^2 t, \lambda x),
\]
which maps a solution $q$ of \eqref{eq:qdNLS} to another solution $q_\lambda$, we see that the scaling critical Sobolev exponent for this equation is
\[
    s_c = -\frac 12.
\]
This suggests that we may expect wellposedness of \eqref{eq:qdNLS} in Sobolev spaces $H^s(\R)$ with $s > -\frac12$. 

To the best of the author's knowledge the strongest result proving local and global wellposedness is the one by Grünrock \cite{gruenrock} for $q_0 \in L^2$, both on the real line and on the torus. This makes the wellposedness situation very similar to the one of cubic NLS before the seminal works of Killip--Visan--Zhang \cite{kvz}, Koch--Tataru \cite{kt} and Harrop-Griffiths--Killip--Visan \cite{hgknv}. They lowered the regularity threshold to $s > -\frac12$, whereas $L^2$ wellposedness was known for a long time. A similar situation was given for the Benjamin-Ono equation, with a-priori estimates being proven in \cite{talbut} and the very recent sharp wellposedness result \cite{klv}. 

The first step on the way to sharp wellposedness in these examples was the construction of conserved quantities at the level of negative Sobolev regularities. This is also the first main result of this work, yielding low regularity a-priori estimates for the equation \eqref{eq:qdNLS} (see Theorem \ref{thm:main1} for the full statement):
\begin{theorem}\label{thm:main1small}
    Given a Schwartz solution $q$ of \eqref{eq:qdNLS}, for all $-\frac12 < s < 0$,
    \begin{equation}
        \|q(t)\|_{H^s} \leq c(1+\|q(0)\|_{H^s})^{\frac{-s}{1+2s}}\|q(0)\|_{H^s}.
    \end{equation}
\end{theorem}

Whether or not wellposedness of \eqref{eq:qdNLS} can be proven in negative regularity similar to \cite{annals,hgknv} is an interesting remaining question (see also Section \ref{sec:questions}). The existence of weak solutions in negative regularity follows from Theorem \ref{thm:main1small} by a compactness argument similar to \cite{kkl}.

To construct the conserved quantities for \eqref{eq:qdNLS} we use its complete integrability. In this situation there are two equivalent formalisms which were presented in the works of Koch--Tataru \cite{kt} and Killip--Visan--Zhang \cite{kvz} for the AKNS hierarchy.

The ansatz of Koch--Tataru \cite{kt} is very classical: they construct the Jost solutions to define the inverse transmission coefficient $T^{-1}$, one of the novelties being their use of Banach spaces of functions of $p$-bounded variation to treat low regularities. Due to the structure of the compatibility condition of AKNS, the transmission coefficient is easily seen to be a conserved quantity. After constructing and showing bounds on the transmission coefficient, one can then analyze its logarithm. By establishing a certain shuffle algebra structure which permits to prove estimates with sharp decay in the spectral parameter, conserved quantities are constructed by integrating $-\log T$ against a weight in the complex plane.

In comparison to the technical but classical machinery in \cite{kt}, the formalism developed in \cite{kvz} is more elegant: Killip--Visan--Zhang start with a certain representation of the transmission coefficient in terms of a (renormalized) Fredholm determinant. Now Jacobi's formula for the logarithm of the determinant allows to get rid of the shuffle algebra used in \cite{kt}, and many estimates on $\log T$ simply reduce to estimates of operators in Schatten class norms. Then again, conserved quantities are constructed by integrating the Fredholm determinant against a weight, respectively by performing a weighted summation.

There is one drawback of the formalism in \cite{kvz} though: one either has to show conservation of the renormalized perturbation determinant by hand, or know that it equals $T^{-1}$ for which conservation is much simpler to prove. In \cite{kvz} the conservation was shown by hand, and it was also shown by hand in the similar situation for Benjamin-Ono \cite{talbut} and dNLS \cite{klaus}. This already indicates one issue: even though the equality of the renormalized Fredholm determinant and the inverse transmission coefficient are mostly assumed to be known (e.g. for dNLS in \cite{bahouri}), apart from rigorous treatments of their equality for KdV (see \cite{jostpais} and \cite[Proposition 5.7]{traceideals}), references for their equality in general higher dimensional systems are harder to find (compare also to \cite{largedataec} where a proof of their equality for dNLS was included). 

On the other hand, showing the conservation of the Fredholm determinant by hand can be a cumbersome task. In fact, for \eqref{eq:qdNLS} the author started with the representation given in \eqref{eq:logdettrace} but struggled to prove its conservation using commutator identities similar to \cite{kvz,klaus}. This is why instead we prove our second main result (see Theorem \ref{thm:main2} for a full statement and Section \ref{sec:equality2} for definitions): 

\begin{theorem}\label{thm:main2small}
    Consider the first order system
    \begin{equation}\label{eq:firstordersystem1}
        (\partial_x - U) \Phi =  0,
    \end{equation}
    where $U = kJ + U_0(u)$, $k > 0$, and
    \begin{enumerate}
        \item $J = \operatorname{diag}(\omega_1, \dots, \omega_n)$, $\omega_j \notin i\R$ are distinct, $\tr(J) = 0$,
        \item $U_0(u)$ is an off-diagonal matrix of polynomials in $u = (u_1,\dots,u_m)$ without zeroth order terms.
    \end{enumerate}
    Then if $\Phi_1^-, \dots, \Phi_1^l$ are the left Jost solutions and $\Phi_{l+1}^+, \dots, \Phi_{n}^+$ are the right Jost solutions, we have the equality
    \begin{equation}
        W\big(\Phi_1^-| \dots |\Phi_l^-|\Phi_{l+1}^+|\dots| \Phi_{n}^+\big) = {\det}_2\big(1-(\partial-kJ)^{-1}U_0(u)\big),
    \end{equation}
    where $W$ denotes the Wronskian and ${\det}_2$ the $2$-renormalized Fredholm determinant.
\end{theorem}

Having established Theorem \ref{thm:main2small} we know that the renormalized Fredholm determinant in \eqref{eq:logdettrace} is conserved, because the transmission coefficient is (which follows from the fact that the matrices in the compatibility condition for \eqref{eq:qdNLS} are trace-free). This conservation is needed to employ the formalism of \cite{kvz}.

The proof of Theorem \ref{thm:main2small} is inspired by work in progress \cite{kkl}. Instead of showing holomorphy of both sides and comparing their zeros as in \cite{traceideals} which in the general situation \eqref{eq:firstordersystem1} has the issue that the zeros are more complicated and not simple anymore (see \cite{bealscoifman}), and instead of writing the equations for Jost solutions with the help of the Fredholm determinant as in \cite{jostpais,largedataec}, we prove the equality of the inverse transmission coefficient and the renormalized Fredholm determinant by showing that their functional derivatives coincide. This may also help in identifying the density function for $\log T$ (compare to the proof of \cite[Proposition 2.4]{annals}) in future work. More precisely we show that with the notation as in Theorem \ref{thm:main2small}, the functional derivative is
\[
    \frac{\delta}{\delta u_i} \log {\det}_2\big(1-(\partial-kJ)^{-1}U_0(u)\big) = \tr(\nabla_i U_0(u)\tilde g),
\]
where $\tilde g$ is the diagonal of the integral kernel of $(\partial_x - U)^{-1} - (\partial_x - kJ)^{-1}$, and $\nabla_i U_0(u)$ is defined by $\frac{d}{ds}|_{s=0}U_0(u+sve_i) = v\nabla_i U_0(u)$ for all $v \in C_c^\infty(\R)$. Compared to \cite{annals}, and as in \cite{hgknv}, the renormalization of $(\partial_x - U)^{-1}$ is necessary before calculating the diagonal Green's function because its first order term $(\partial_x - kJ)^{-1}$ does not have a meaningful restriction to $x = y$. 

This work is structured as follows. In Section \ref{sec:preliminaries} we list preliminaries to understand later calculations. In Section \ref{sec:transmission} we construct the Jost solutions and the transmission coefficient for \eqref{eq:qdNLS}. This is also done in the case when the function has less than $L^2$ regularity and lies in a modulation space (see Theorem \ref{thm:modulation}). In Section \ref{sec:equality1} we prove equality of the inverse transmission coefficient and the renormalized Fredholm determinant for \eqref{eq:qdNLS}. These proofs are less involved than in the general case and help in its understanding.  In Section \ref{sec:apriori} we prove Theorem \ref{thm:main1} and in Section \ref{sec:equality2} we prove Theorem \ref{thm:main2}. Section \ref{sec:questions} lists some open and interesting questions which the author will probably not address soon.

\subsubsection*{Acknowledgements} The author is grateful to Herbert Koch for useful discussions relating to Lemma \ref{lem:formofgreensfunctiongeneral}. He also in general profited from many ideas from \cite{kkl}. Funded by the Deutsche Forschungsgemeinschaft (DFG, German Research Foundation) – Project-
ID 258734477 – SFB 1173.

\section{Preliminaries}\label{sec:preliminaries}

In what follows we will often make use of the operators $(\partial - z)^{-1}$, which we define for $f \in \Sc'(\R)$ as
\begin{equation}\label{eq:integralkernel}
    ((\partial - z)^{-1}f)(x) = \begin{cases}-\int_x^\infty e^{z(x-y)}f(y)\, dy, \qquad \text{if } \real z > 0,\\
        \int_{-\infty}^x e^{z(x-y)}f(y)\, dy, \qquad \text{ if } \real z <  0.\end{cases}
\end{equation}
In particular when assuming $k > 0$ the operators $R_1 = (\partial- k)^{-1}, R_2 = (\partial- \omega^2k)^{-1}, R_3 = (\partial- \omega k)^{-1}$, $\omega = \exp(2\pi i/3)$, have integral kernels 
\begin{align*}
    K_1(x,y) &= - \chi_{\{x < y\}} e^{k(x-y)},\\
    K_2(x,y) &= \chi_{\{x > y\}} e^{\omega^2 k(x-y)},\\
    K_3(x,y) &= \chi_{\{x > y\}} e^{\omega k(x-y)}.
\end{align*}
Given a functional $F: L^2 \to \C$ we define its functional derivative $\frac{\delta F}{\delta u}$ via duality as the unique function resp. distribution satisfying
\[
    \frac{d}{ds}\Big|_{s = 0} F(u + sv) = \int_\R \frac{\delta F}{\delta u}(x)v(x)\, dx, \qquad \text{for all } v \in \Sc(\R).
\]
Partial functional derivatives are defined similarly. We denote the Schatten class of compact operators with $p$ summable singular values as $\I_p$. $\I_1$ is called trace class, $\I_2$ are Hilbert-Schmidt operators and $\I_\infty$ are the compact operators. If $A \in \I_1$ is of trace class on the separable Hilbert space $H$, then its trace is
\[
    \tr_H(A) = \sum_{n=1}^\infty \langle \phi_n, A\phi_n\rangle, \qquad (\phi_n)_n \text{ an orthonormal basis}.
\]
In particular if $B$ which is written in matrix form is trace class on $H^n$,
\[
    B = \left(\begin{matrix}A_{11} & \dots & A_{1n}\\
    \dots & \dots & \dots \\
    A_{n1} & \dots & A_{nn}\end{matrix}\right),
\]
its trace is
\begin{equation}\label{eq:higherdimensionaltrace}
    \tr_{H^n}(B) = \sum_{i=1}^n \tr_H(A_{ii}).
\end{equation}
Hilbert-Schmidt operators on $L^2(\R)$ can be identified with their integral kernels $K(x,y) \in L^2(\R^2)$. Schatten class operators satisfy the Hölder-type estimate
\[
    |\tr(AB)| \leq \|AB\|_{\I_1} \leq \|A\|_{\I_p}\|B\|_{\I_q}, \qquad \frac1p + \frac1q = 1,
\]
and the continuous embeddings
\[
    \I_p \subset \I_q, \qquad p \leq q,
\]
hold. Given $A \in \I_1$ one can define the operator determinant of $1 + A$ by the formula
\[
    \log \det(1+A) = \sum_{n=1}^\infty \frac{(-1)^{n+1}}{n}\tr(A^n),
\]
assuming that the latter sum is absolutely convergent. This formula is called Jacobi's formula. If $A$ is a matrix of operators we understand the traces as in \eqref{eq:higherdimensionaltrace}. Given $A \in \I_k$, define the $k$-regularized determinant of $1+A$ by
\[
    \log {\det}_n(1+A) = \sum_{n=k}^\infty \frac{(-1)^{n+1}}{n}\tr(A^n).
\]
We refer to \cite{simon} for a complete introduction to these topics.

\section{The transmission coefficient}\label{sec:transmission}

Equation \eqref{eq:qdNLS} can be seen as a system of equations
\begin{equation}\label{eq:system}
    \begin{cases}
        iq_t + \frac1{\sqrt{3}}q_{xx} + 2irr_x = 0,\\
        ir_t - \frac1{\sqrt{3}}r_{xx} + 2iqq_x = 0,
    \end{cases}
\end{equation}
with the restriction $q = \bar r$. Since the arguments are the same and the calculations are a bit more lucid, we will treat the system \eqref{eq:system} in place of \eqref{eq:qdNLS} until the construction of almost conserved quantities in Section \ref{sec:apriori}. We can write the system \eqref{eq:system} as a Hamiltonian system. To this end we define
\[
    \{F,G\} = \int \frac{\delta F}{\delta q}\Big(\frac{\delta G}{\delta r}\Big)' + \frac{\delta F}{\delta r}\Big(\frac{\delta G}{\delta q}\Big)'.
\]
We can then write Hamilton's equations
\[
    q_t = \{q,H\}, \qquad r_t = \{r,H\}
\]
as
\begin{equation}\label{eq:Hamiltoniansystem}
    q_t = \Big(\frac{\delta H}{\delta r}\Big)', \qquad r_t = \Big(\frac{\delta H}{\delta q}\Big)'.
\end{equation}
Now if
\[
    H = \int \frac{i}{2\sqrt3}\big(q'r-qr') - \frac13\big(r^3 + q^3\big),
\]
we recover \eqref{eq:system} from the Hamiltonian equations \eqref{eq:Hamiltoniansystem}. The system \eqref{eq:system} admits a Lax pair with a $3\times 3$ Lax operator (see \eqref{eq:laxoperatorfromU})
\begin{equation}\label{eq:L}
    L = \left(\begin{matrix}\partial - k & -q & -r\\-r & \partial - \omega^2 k & -q\\ -q & -r & \partial - \omega k  \end{matrix}\right).
\end{equation}
Here, $\omega = \operatorname{Exp}(2\pi i/3)$ is a third root of unity.

The qdNLS equations \eqref{eq:qdNLS} can be written as the compatibility condition of the system of equations
\begin{equation}\label{eq:compatibility}
    \begin{split}
    \psi_x &= U \psi,\\
    \psi_t &= V \psi
    \end{split}
\end{equation}
where the matrices $U$ and $V$ are (see \cite{lenells})
\begin{align*}
    U = k\left(\begin{matrix} 1&0&0\\0&\omega^2&0\\0&0&\omega\end{matrix}\right) + \left(\begin{matrix} 0&q&r\\r&0&q\\q&r&0\end{matrix}\right) = kJ + U_0, \qquad \omega = e^{\frac{2\pi i }{3}},
\end{align*}
and
\begin{align*}
    V = -k^2\left(\begin{matrix} 1&0&0\\0&\omega&0\\0&0&\omega^2\end{matrix}\right) + k\left(\begin{matrix} 0&\omega q&\omega^2 r\\\omega r&0&q\\\omega^2 q&r&0\end{matrix}\right) + \left(\begin{matrix} 0&\frac{i}{\sqrt{3}}q_x - r^2&-\frac{i}{\sqrt{3}}r_x-q^2\\-\frac{i}{\sqrt{3}}r_x-q^2 &0&\frac{i}{\sqrt{3}}q_x-r^2\\ \frac{i}{\sqrt{3}}q_x-r^2&-\frac{i}{\sqrt{3}}r_x-q^2&0\end{matrix}\right).
\end{align*}
For the overdetermined system \eqref{eq:compatibility} to be solvable the compatibility condition
\[
    U_t - V_x + [U,V] = 0
\]
has to be satisfied. The Lax operator $L$ can be recovered from this as
\begin{equation}\label{eq:laxoperatorfromU}
    L = \partial_x - U,
\end{equation}
since in this case
\[
    L_t = -U_t = [\partial_x - L,V] - V_x = [V,L].
\]

We are interested in the inverse transmission coefficient and thus the Jost solutions of $L$. For simplicity we suppress the time dependence and assume $q,r \in C_c^\infty(\R)$ for the moment. Let $K \subset \R$ be a compact interval with $\supp q \cup \supp r \subset K$. Then if $x \in K^c$ is outside of the support of $q$ and $r$, a solution of
\begin{equation}\label{eq:ODE}
    L\phi = 0
\end{equation}
is a solution of
\begin{equation}\label{eq:asymptoticequality}
    \phi_x = k\left(\begin{matrix} 1&0&0\\0&\omega^2&0\\0&0&\omega\end{matrix}\right)\phi.
\end{equation}
We can thus define two systems of solutions of $L\psi = 0$, defined by asymptotics at $x \to \pm \infty$ (which is the same as $x < \inf K$ and $x > \sup K$, respectively, in this case),
\begin{equation}\label{eq:Jostasymptotics}
    \Phi_1^{\pm} = \left(\begin{matrix} e^{kx}\\0\\0\end{matrix}\right), \quad \Phi_2^{\pm} = \left(\begin{matrix} 0\\e^{\omega^2 k x}\\0\end{matrix}\right), \quad \Phi_3^{\pm} = \left(\begin{matrix} 0\\0\\e^{\omega k x}\end{matrix}\right).
\end{equation}
Note that the existence of these functions is ensured by the Cauchy-Lipschitz resp. Picard-Lindelöf theorem, by solving the ODE \eqref{eq:ODE} starting from $x_0 = \inf K$ for $\Phi_i^-$, resp. starting from $x_0 = \sup K$ and solving backwards in $x$ for $\Phi_i^+$.

Recall that $\omega = 1/2(-1 + \sqrt{3}i)$. We will only consider real-valued and positive $k > 0$ when dealing with \eqref{eq:L}, thus the three functions $\Phi_1^-, \Phi_2^+, \Phi_3^+$ are decaying at their corresponding defining asymptotics. We call $\Phi_1^-$ left Jost solution and both $\Phi_2^+, \Phi_3^+$ right Jost solutions. 

Usually one has to be a bit careful with the definition of Jost solutions in higher-order systems. In the case of compactly supported potentials, the definition \eqref{eq:Jostasymptotics} is unambiguous because the asymptotics \eqref{eq:Jostasymptotics} are an actual equality if $x < \inf K$ respectively $x > \sup K$. If we want to extend the definition to Schwartz potentials though, we are lead to define Jost solutions by the condition
\begin{equation}\label{eq:Jostasymptotics2}
    \lim_{x \to \pm \infty} e^{-kx}\Phi_1^{\pm}(x) = \left(\begin{matrix} 1\\0\\0\end{matrix}\right), \quad \lim_{x \to \pm\infty}e^{-\omega^2 kx}\Phi_2^{\pm}(x) = \left(\begin{matrix} 0\\1\\0\end{matrix}\right), \quad \lim_{x\to \pm\infty}e^{-\omega k x}\Phi_3^{\pm}(x) = \left(\begin{matrix} 0\\0\\1\end{matrix}\right).
\end{equation}
Now suppose that $k \in \C \setminus \R$ and that $\real(k\omega) < \real(k\omega^2)$, and that $\Phi_3^+$ is (uniquely) given. Then if $\Phi_2^+$ satisfies the asymptotics in \eqref{eq:Jostasymptotics2} we can add any multiple of $\Phi_3^+$ to $\Phi_2^+$ and the result still satisfies the same asymptotics as $\Phi_2^+$. Thus we face a problem in the uniqueness of the definition of Jost solutions.

In this specific example the uniqueness issue can be overcome by also requiring that $e^{-\omega^2 kx}\Phi_2^{\pm}(x) \in L^\infty(\R)$, as was outlined in the work of Beals-Coifman \cite{bealscoifman}. In here we instead simply restrict to $k > 0$. This results in $\real(k\omega) = \real(k\omega^2)$, making the asymptotics \eqref{eq:Jostasymptotics2} unique (and also bringing us into the excluded set $\Sigma$ in the language of \cite{bealscoifman}).

We define the components $\Phi_{ij}^\pm$ as the entries of the matrix
\[
    \Phi^{\pm} = \big(\Phi_1^\pm|\Phi_2^\pm|\Phi_3^\pm\big),
\]
so that e.g. $\Phi_1^- = (\Phi_{11}^-, \Phi_{21}^-, \Phi_{31}^-)^t$. Note that due to the linear independence of their columns, both matrices $\Phi^-$ and $\Phi^+$ are fundamental matrices for the problem \eqref{eq:ODE}.

Using the Jost solutions we define the inverse of the transmission coefficient as the Wronskian
\begin{equation}\label{def:T}
    T^{-1}(k,q,r) := W(\Phi_1^-,\Phi_2^+,\Phi_3^+),
\end{equation}
where the Wronskian is
\[
    W(\psi_1,\psi_2,\psi_3) = \det(\psi_1|\psi_2|\psi_3).
\]
This definition is independent of $x$: a small calculation reveals that for all $v_1,v_2,v_3 \in \C^3, A \in \C^{3\times 3}$,
\begin{equation}\label{eq:dettotrace}
        \det(Av_1|v_2|v_3) + \det(v_1|Av_2|v_3)+\det(v_1|v_2|Av_3) = \tr(A)\det(v_1|v_2|v_3),
\end{equation}
and hence, 
\[
    \partial_x W(\Phi_1^-,\Phi_2^+,\Phi_3^+) = \tr(U) W(\Phi_1^-,\Phi_2^+,\Phi_3^+) = 0.
\]
The inverse transmission coefficient can also be written as a limit of the left Jost solution at $+\infty$,
\begin{equation}\label{def:Tlimit}
    T^{-1}(k,q,r) = \lim_{x \to \infty} e^{-kx}\Phi_{11}^-(k,q,r).
\end{equation}
Indeed, since the Wronskian is independent of $x$ we see that for all $x > \sup K$
\[
    T^{-1} = W\left(\begin{matrix} \Phi_{11}^-(x) & 0 & 0\\ \Phi_{21}^-(x) & e^{\omega^2 kx} & 0\\ \Phi_{31}^-(x) & 0 & e^{\omega k x}\end{matrix}\right) = e^{-kx}\Phi_{11}^{-}(x),
\]
and thus the limit on the right-hand exists for $x \to \infty$ and also equals $T^{-1}$.

For the simplicity of calculations we often want to use equations like \eqref{eq:qdNLS} in a point-wise fashion. Unfortunately, solutions of dispersive equations with $C^\infty_c$ initial data tend to lose the compact support immediately as time progresses due to the infinite speed of propagation. The better alternative is then to speak of Schwartz solutions of \eqref{eq:qdNLS}, and we mean by that solutions $q \in C^\infty(\R,\Sc(\R))$ of \eqref{eq:qdNLS}. These solutions exist, as can be seen from the work of Beals--Coifman \cite{bealscoifman} respectively by using a persistence of regularity argument on the global solutions of Grünrock \cite{gruenrock}.

Thus we have to extend the definition of Jost solutions and the inverse transmission coefficient from $C^\infty_c$ functions to Schwartz functions. This is done in the next Theorem \ref{thm:jostsolutionslowreg} where it is actually shown that we need much less than Schwartz regularity and decay. 

Note though that this theorem restricts to the left Jost solution, which by \eqref{def:Tlimit} is fortunately enough to define the inverse transmission coefficient. For constructing both right Jost solutions, $L^2$ decay will in general be not enough due to $\real(\omega k) = \real(\omega^2 k)$, instead one would have to impose an $L^1$ condition. 

\begin{theorem}\label{thm:jostsolutionslowreg}
    There exists $\delta > 0$ such that for all $q,r \in L^2$ and $k > 0$ with $k^{-1/2}(\|q\|_{L^2} + \|r\|_{L^2}) \leq \delta$ there exists a unique left Jost solution $\Phi_1^-$ of the equation
    \[
        \psi_x = U\psi.
    \]
    It satisfies $e^{-kx}\Phi_1^-(x) \in C^0(\R) \times (C^0 \cap L^2(\R)) \times (C^0 \cap L^2(\R))$. Moreover, for $k > 0$ the map
    \[
        \{\|q\|_{L^2} + \|r\|_{L^2} \leq k^{1/2}\delta \} \subset L^2(\R)^2 \to \C, \qquad (q,r) \mapsto T^{-1}(k,q,r)
    \]
    is analytic.
\end{theorem}

\begin{proof}
    We construct $\Phi_1^-$. We are looking for a solution of
    \[
        \psi_x = \left(\begin{matrix}k&q&r\\r&k\omega^2&q\\q&r&k\omega\end{matrix}\right) \psi
    \]
    with asymptotics
    \[
        \psi(x) = \left(\begin{matrix}e^{kx}\\0\\0\end{matrix}\right) + o(1),
    \]
    as $x \to - \infty$. Let $\phi = e^{-kx}\psi(x)$. Then $\phi$ satisfies,
    \begin{align*}
        \phi_{1,x} &= q\phi_2 + r\phi_3,\\
        \phi_{2,x} &= r\phi_1 + k(\omega^2 -1) + q\phi_3,\\
        \phi_{3,x} &= q\phi_1 + r\phi_2 + k(\omega - 1)\phi_3,
    \end{align*}
    and hence
    \begin{equation}\label{eq:systemphi}
    \begin{split}
        \phi_1 &= \partial_x^{-1}(q\phi_2 + r\phi_3) + 1,\\
        \phi_2 &= (\partial_x -k(\omega^2-1))^{-1}(r\phi_1 + q\phi_3),\\
        \phi_3 &= (\partial_x -k(\omega-1))^{-1}(q\phi_1 + r\phi_2),
    \end{split}
    \end{equation}
    where we define the operator $\partial_x^{-1}$ acting on $L^1$ functions by
    \[
        (\partial_x^{-1}f)(x) = \int_{-\infty}^x f(y)\, dy,
    \]
    and the operators $(\partial_x - \eta)^{-1}$ as in \eqref{eq:integralkernel}. Note that the numbers $k(\omega^2-1)$ and $k(\omega-1)$ have negative real part. In particular by Young's convolution inequality,
    \begin{align}
        \label{eq:est1}\|\partial^{-1}f\|_{L^\infty} &\leq \|f\|_{L^1}\\
        \label{eq:est2}\|(\partial-k(\omega^2-1))^{-1}f\|_{L^p}+\|(\partial-k(\omega-1))^{-1}f\|_{L^p} &\lesssim k^{-1 - \frac1p+\frac1q}\|f\|_{L^q}, \qquad q \leq p
    \end{align}
    We can solve the system \eqref{eq:systemphi} via Picard iteration in the Banach space $X_k = C^0_0 \times (C^0_0 \cap L^2) \times (C^0_0 \cap L^2)$, where $C^0_0$ denotes continuous functions with limit at $\pm \infty$. Note that $C^0_0$ is a closed subspace of $L^\infty(\R)$. We endow $X_k$ with the norm
    \[
        \|\phi\|_{X_k} = \|\phi_1\|_{L^\infty} +\|\phi_2\|_{L^\infty} +\|\phi_3\|_{L^\infty} + k^{\frac12}(\|\phi_2\|_{L^2} + \|\phi_3\|_{L^2}).
    \]
    We write \eqref{eq:systemphi} as
    \begin{align*}
        \phi = S\phi + \left(\begin{matrix}1\\0\\0\end{matrix}\right)
    \end{align*}
    where
    \begin{align*}
        S = \left(\begin{matrix}0&\partial^{-1}q&\partial^{-1}r\\(\partial-k(\omega^2-1))^{-1}r&0&(\partial-k(\omega^2-1))^{-1}q\\(\partial-k(\omega-1))^{-1}q&(\partial-k(\omega-1))^{-1}r&0\end{matrix}\right)
    \end{align*}
    Then \eqref{eq:est1} and \eqref{eq:est2} show that
    \[
        \|S\phi\|_{X_k} \lesssim k^{-1/2}(\|q\|_{L^2}+\|r\|_{L^2})\|\phi\|_{X_k},
    \]
    hence by writing the Neumann series for $(1-S)^{-1}$ there exists a unique solution $\phi \in X_k$ solving \eqref{eq:systemphi} provided
    \[
        k^{-1/2}(\|q\|_{L^2}+\|r\|_{L^2}) \ll 1.
    \]
    The analyticity of the map $(q,r) \mapsto T^{-1}$ now follows from \eqref{def:Tlimit}.
\end{proof}

The importance of the transmission coefficient is that it is a conserved quantity for solutions of \eqref{eq:qdNLS}:

\begin{theorem}\label{thm:conservation}
    Given $k > 0$ and $q,r$ Schwartz solutions of \eqref{eq:qdNLS}, the transmission coefficient $T^{-1}(k,q,r)$ is conserved.
\end{theorem}
\begin{proof}
    If $q,r$ are Schwartz solutions of equations \eqref{eq:system} then the matrices from \eqref{eq:compatibility} satisfy
    \[
        U_t - V_x + [U,V] = 0.
    \]
    Thus given a solution $\psi$ of $\psi_x = U\psi$ we have
    \[
        (\psi_t - V\psi)_x = U(\psi_t - V\psi).
    \]
    This means that $\psi_t - V\psi$ can be written as a linear combination of $(\Phi_i^{\pm})_{i=1,2,3}$. From the time independence of the asymptotics of $\Phi_{1}^-$ at $- \infty$, and the decaying behaviour of $q,r$ in the matrix $V$, we find
    \begin{align*}
        \Phi_{1,t}^- = V\Phi_{1}^- + k^2 \Phi_{1}^-.
    \end{align*}
    Thus by \eqref{eq:dettotrace},
    \begin{align*}
        \partial_t T^{-1} &= \Big[\lim_{x\to \infty} e^{kx}(V\Phi_{1}^- + k^2 \Phi_{1}^-)(x)\Big]_1
        = \left[-k^2 J^2 \left(\begin{matrix}T^{-1}\\0\\0\end{matrix}\right) + k^2\left(\begin{matrix}T^{-1}\\0\\0\end{matrix}\right)\right]_1 = 0.
    \end{align*}
    Here we were able to exchange limit and $\partial_t$ because $\lim_{x\to \infty}$ is a linear operator acting on the space $C^0_0$ and $t \mapsto \Phi_1^-$ is continuously differentiable in this space.
\end{proof}

In order to connect the transmission coefficient to the renormalized Fredholm determinant we want to use their respective functional derivatives. For the inverse transmission coefficient, these are calculated in the next lemma.

\begin{lemma}\label{lem:functionalderivative}
    Let $q,r \in C_c^\infty(\R)$. The functional derivatives of the transmission coefficient are
    \begin{align*}
        \frac{\delta T^{-1}}{\delta q} &= \Phi_{11}^-(\Phi_{12}^+\Phi_{23}^+ - \Phi_{13}^+\Phi_{22}^+) + \Phi_{21}^-(\Phi_{22}^+\Phi_{33}^+ - \Phi_{23}^+\Phi_{32}^+) + \Phi_{31}^-(\Phi_{13}^+\Phi_{32}^+ - \Phi_{12}^+\Phi_{33}^+)\\
        \frac{\delta T^{-1}}{\delta r} &= \Phi_{11}^-(\Phi_{13}^+\Phi_{32}^+ - \Phi_{12}^+\Phi_{33}^+) + \Phi_{21}^-(\Phi_{12}^+\Phi_{23}^+ - \Phi_{13}^+\Phi_{22}^+) + \Phi_{31}^-(\Phi_{22}^+\Phi_{33}^+ - \Phi_{23}^+\Phi_{32}^+)
    \end{align*}
\end{lemma}
\begin{proof}
    We use that
    \[
        T^{-1} = \lim_{x \to \infty} e^{-kx}\Phi^-_{11}(x).
    \]
    Given $q,r$ compactly supported, the fundamental matrix $\Phi^-$ satisfies
    \[
        \Phi_x = U\Phi, \qquad \Phi(x) = \left(\begin{matrix}
            e^{kx}&0&0\\0&e^{\omega^2 kx}&0\\0&0&e^{\omega kx}
        \end{matrix}\right)\quad \text{if}\quad x \ll 0.
    \]
    Thus $\dot\Phi^-(k,q,r) = \frac{d}{ds}|_{s=0} \Phi^-(k,q+s\tilde q,r+s\tilde r)$ solves
    \[
        \dot \Phi_x = U\dot \Phi + \dot U\Phi, \qquad \Phi(x) = 0\quad \text{if}\quad x \ll 0,
    \]
    and where
    \begin{equation*}
    \dot U = \left(\begin{matrix}0 & \tilde q & \tilde r\\ \tilde r & 0 & \tilde q\\ \tilde q & \tilde r & 0  \end{matrix}\right).
    \end{equation*}
    By the variation of constants formula, the unique solution of this boundary value problem can be written as
    \[
        \dot \Phi^-(x) = \int_{-\infty}^x \Phi^+(x)\big(\Phi^+(y)\big)^{-1}\dot U(y)\Phi^-(y)\, dy,
    \]
    where integrability towards $-\infty$ is no issue due to the compact support of $q,r$. Here the condition $\dot \Phi^-(x) = 0$ if $x \ll 0$ is ensured by the integration up to $x$ and the compact support. Now note that $e^{(\omega^2-1)kx}, e^{(\omega^2-1)kx} \to 0$ as $x \to \infty$, hence
    \[
        \lim_{x \to \infty} e^{-kx}\Phi^+(x) = \left(\begin{matrix}
            1&0&0\\0&0&0\\0&0&0
        \end{matrix}\right) =: E_{11}.
    \]
    Thus,
    \[
        \dot T^{-1} = \int_\R \Big[E_{11}\big(\Phi^+(y)\big)^{-1}\dot U(y)\Phi^-(y)\, \Big]_{11} \,dy.
    \]
    Now by Jacobi's formula for the inverse of a matrix, and using that $\det \Phi^+(y) = 1$ for all $y \in \R$,
    \[
        E_{11}\big(\Phi^+\big)^{-1} = \left(\begin{matrix}
            \Phi^+_{22}\Phi^+_{33}-\Phi^+_{23}\Phi^+_{32}&\Phi^+_{13}\Phi^+_{32}-\Phi^+_{12}\Phi^+_{33}&\Phi^+_{12}\Phi^+_{23}-\Phi^+_{13}\Phi^+_{22}\\ 0&0&0\\0&0&0
        \end{matrix}\right),
    \]
    and moreover
    \[
        \dot U \Phi^- = \left(\begin{matrix}
            \tilde q \Phi^-_{21} + \tilde r \Phi^-_{31}&\dots &\dots\\ \tilde q \Phi^-_{31} + \tilde r \Phi^-_{11}&\dots &\dots\\ \tilde q \Phi^-_{11} + \tilde r \Phi^-_{21}&\dots &\dots
        \end{matrix}\right),
    \]
    Hence
    \begin{align*}
        \dot T^{-1} &= \int_\R (\tilde q \Phi^-_{21} + \tilde r \Phi^-_{31})(\Phi^+_{22}\Phi^+_{33}-\Phi^+_{23}\Phi^+_{32}) +(\tilde q \Phi^-_{31} + \tilde r \Phi^-_{11})(\Phi^+_{13}\Phi^+_{32}-\Phi^+_{12}\Phi^+_{33}) \\
        &\qquad + (\tilde q \Phi^-_{11} + \tilde r \Phi^-_{21})(\Phi^+_{12}\Phi^+_{23}-\Phi^+_{13}\Phi^+_{22}) \,dy,
    \end{align*}
    finishing the proof of Lemma \ref{lem:functionalderivative}.
\end{proof}

Before continuing with the renormalized Fredholm determinant we want to comment on the $L^2$ assumption in Theorem \ref{thm:jostsolutionslowreg}. While for our purposes (namely defining $T^{-1}$ for Schwartz solutions) $L^2$ was clearly enough, the question of a threshold regularity to define Jost solutions is still very interesting. Approaches to define the Jost solutions in lower regularity $H^s(\R)$, $-\frac12 < s < 0$ were laid out by Koch--Tataru \cite{kt}. They used a special class of spaces of bounded variation in order to perform a fixed point iteration similar to Theorem \ref{thm:jostsolutionslowreg}. 

There is also a different approach, making use of so called modulation spaces $M_{2,p}(\R)$, which seems to be new. Define the modulation space norm as
\[
    \|q\|_{M_{2,p}} = \Big(\sum_{n \in \Z}\Big(\int_{n}^{n+1} |\hat q(\xi)|^2 \, d\xi\Big)^\frac{p}{2}\Big)^{\frac1p},
\]
and $M_{2,p}(\R)$ as the subspace of tempered distributions with finite modulation space norm.
While compared to the spaces $l^2_\tau DU^2$ used in \cite{kt} modulation spaces do not admit an embedding of the type\[
    H^{-\frac12 +\varepsilon}(\R) \subset X, \qquad \varepsilon > 0,
\]
they still comprise functions of lower regularity than $L^2(\R)$ since
\[
    L^2(\R) \subset M_{2,p} \subset H^{-\frac12}(\R), \qquad \text{ if } p < \infty.
\]
There are also versions of these spaces with $L^r$ type decay, $M_{r,p}(\R)$. Similar to before these can be defined in terms of a norm,
\[
    \|q\|_{M_{r,p}} = \Big(\sum_{n \in \Z}\|\square_n q\|_{L^r}^p\Big)^{\frac1 p},
\]
where $(\square_n)_{n \in \Z}$ are the Fourier multiplier operators corresponding to a smooth, positive, unit sized partition of unity. The spaces $M_{r,1}$ used below can be shown to satisfy the continuous embeddings
\[
    M_{r,1}(\R) \subset C^0_b(\R) \cap L^r(\R),
\]
where $C^0_b(\R)$ denotes the space of continuous and bounded functions. In particular $M_{r,1} \cap C^0_0$ is a closed subspace of $M_{r,1}$.

We do not go into details here and refer to \cite[Section 2]{klaus2} for a concise overview, respectively the book \cite{wang} for a more thorough introduction to these spaces.

\begin{theorem}\label{thm:modulation}
    There exists $\delta > 0$ such that for all $k > 0$, and $q,r \in M_{2,p}$, $p \in [1,\infty)$ with $k^{-\frac1p}(\|q\|_{M_{2,p}} + \|r\|_{M_{2,p}}) \leq \delta$ there exists a unique left Jost solution of the equation
    \[
        \psi_x = U\psi
    \]
    with asymptotics \eqref{eq:Jostasymptotics}. It satisfies $e^{-kx}\Phi_1^- \in M_{\infty,1} \times M_{2,1} \times M_{2,1}$. Moreover, for $k > 0$ the map
    \[
        \{\|q\|_{M_{2,p}} + \|r\|_{M_{2,p}} \leq k^{\frac1p}\delta \} \subset (M_{2,p})^2 \to \C, \qquad (q,r) \mapsto T^{-1}(k,q,r)
    \]
    is analytic.
\end{theorem}

\begin{proof}
    We are again reduced to the system \eqref{eq:systemphi}, but we can only estimate $q,r \in M_{2,q}$. We claim first that if $R = (\partial - k(\omega^2 - 1))^{-1}$, the estimate
    \[
        \|R\|_{M_{2,p} \to M_{2,1}} \lesssim_p k^{-\frac1p}
    \]
    holds. Now indeed, with $\omega^2 - 1 = - \frac32 - \frac{\sqrt{3}}{2}i = a + ib$,
    \begin{align*}
        \|Rf\|_{M_{2,1}} &= \sum_{n \in \Z}\Big(\int_n^{n+1} \frac{|\hat f(\xi)|^2}{(\xi+kb)^2 + k^2a^2}\Big)^{\frac12} \approx \sum_{n \in \Z}(n^2 + k^2a^2)^{-\frac12} \Big(\int_n^{n+1} |\hat f(\xi-kb)|^2\Big)^{\frac12}\\
        &\lesssim \Big(\sum_{n \in \Z}(n^2 + k^2a^2)^{-\frac{p'}2}\Big)^{\frac1{p'}}\Big(\sum_{n \in \Z}\Big(\int_n^{n+1} |\hat f(\xi-kb)|^2\Big)^{\frac p2}\Big)^{\frac1p}\\
        &\lesssim_p k^{-\frac1p}\|f\|_{M_{2,p}},
    \end{align*}
    where in the second line we estimated with Hölder's inequality, and in the third line we calculated the sum and estimated the characteristic function on $[n - kb,n+1-kb]$ by two unit sized characteristic functions supported between the surrounding integers. Thus,
    \[
        \|R(r\phi_1 + q\phi_3)\|_{M_{2,1}} \lesssim k^{-\frac1p}(\|r\|_{M_{2,p}} + \|q\|_{M_{2,p}})(\|\phi_1\|_{M_{\infty,1}} + \|\phi_3\|_{M_{2,1}})
    \]
    where we used both the Hölder type estimate
    \[
        \|fg\|_{M_{r,p}} \lesssim \|f\|_{M_{r_1,p_1}}\|g\|_{M_{r_2,p_2}}, \qquad \frac1p = \frac1{p_1} + \frac1{p_2}, 1 + \frac1{q} = \frac1{q_1} + \frac1{q_2}
    \]
    as well as the continuous embedding $M_{2,1} \subset M_{\infty,1}$. In a similar fashion we estimate
    \[
        \|(\partial - k(\omega - 1)^{-1}(q\phi_1 + r\phi_2)\|_{M_{2,1}} \lesssim k^{-\frac1p}(\|r\|_{M_{2,p}} + \|q\|_{M_{2,p}})(\|\phi_1\|_{M_{\infty,1}} + \|\phi_2\|_{M_{2,1}}).
    \]
    It remains to show that for all $p < \infty$,
    \begin{equation}\label{eq:modulationintegration}
        \Big\|\int_{-\infty}^x f(y)\, dy\Big\|_{M_{\infty,1}} \lesssim \|f\|_{M_{1,p}}.
    \end{equation}
    Indeed, then \eqref{eq:modulationintegration} implies that
    \[
        \|\partial^{-1}(q\phi_2 + r\phi_3)\|_{M_{\infty,1}} \lesssim (\|r\|_{M_{2,p}} + \|q\|_{M_{2,p}})(\|\phi_2\|_{M_{2,1}} + \|\phi_3\|_{M_{2,1}})
    \]
    after another use of the Hölder-type embedding, thus we can iterate in the Banach space $X_k = (M_{\infty,1} \cap C^0_0) \times (M_{2,1} \cap C^0_0) \times (M_{2,1} \cap C^0_0)$ with norm
    \[
        \|\phi\|_{X_k} = \|\phi_1\|_{M_{\infty,1}} +k^{\frac1p}(\|\phi_2\|_{M_{2,1}} +\|\phi_3\|_{M_{2,1}})
    \]
    similar to Theorem \ref{thm:jostsolutionslowreg}.

    To prove \eqref{eq:modulationintegration} we split into the frequency around zero,
    \[
        \|\partial_x^{-1}\square_0 f\|_{L^\infty} \leq \|\square_0 f\|_{L^1},
    \]
    which is a simple consequence of the triangle inequality, and the non-zero frequencies,
    \begin{align*}
        \sum_{n \neq 0} \|\partial_x^{-1}\square_n f\|_{L^\infty} &\leq \sum_{n \neq 0} \|\partial_x^{-1}\square_n f\|_{L^2} \\
        &= \sum_{n \neq 0} \Big(\int_n^{n+1} \frac{|\hat f(\xi)|^2}{\xi^2}\, d\xi\Big)^{\frac12}\\
        &\approx \sum_{n \neq 0} |n|^{-1}\|\square_n f\|_{L^2}\\
        &\lesssim_p \Big(\sum_{n \neq 0} \|\square_n f\|_{L^1}^p\Big)^{\frac1p},
    \end{align*}
    where in the first line we used Bernstein's inequality, in the second line we used that for $g \in \Sc(\R)$ with $\hat g(\xi)$ supported away from zero one has $\mathcal{F}(\partial^{-1}g)(\xi) = \hat g(\xi)/(i\xi)$, in the third line we approximated $\xi \approx n$ and in the last line we used Bernstein and Hölder.
\end{proof}

\section{The Transmission Coefficient as a renormalized Fredholm determinant}\label{sec:equality1}

In this section we prove that the transmission coefficient has a representation as a renormalized Fredholm determinant.

To do so we need some notation. Given the Lax operator $L$,
\[
    L(k,q,r) = \left(\begin{matrix}\partial - k & -q & -r\\-r & \partial - \omega^2 k & -q\\ -q & -r & \partial - \omega k  \end{matrix}\right),
\]
we define $L_0(k) = L(k,0,0) = \partial - kJ$, and its inverse $R_0 = L_0^{-1}$. The experience with KdV, NLS, mKdV and dNLS (\cite{kvz,klaus}) tells us that we should look at the Fredholm determinant
\[
    \log {\det} (1+R_0(L-L_0)) = \log {\det} (1-R_0U_0).
\]
Unfortunately, the operator
\[
    \tilde \Lambda = R_0(L-L_0) = -\left(\begin{matrix}0 & R_1q & R_1r\\R_2r & 0 & R_2q\\ R_3q & R_3r & 0  \end{matrix}\right),
\]
where
\[
    R_1 = (\partial-k)^{-1}, \qquad R_2 = (\partial-\omega^2k)^{-1}, \qquad R_3 = (\partial-\omega k)^{-1},
\]
is not trace-class, and Hilbert-Schmidt only in the case of $q,r \in L^2$. The first issue can be dealt with by looking at the renormalization of the Fredholm determinant (and because formally at least, $\tr(\tilde \Lambda) = 0$, so that this renormalization gives formally the same quantity), whereas for the second issue we use that $\tilde \Lambda$ is equivalent to $\Lambda = \sqrt{R_0}(L-L_0)\sqrt{R_0}$. Thus we define
\begin{equation}\label{eq:defA}
    A(k,q,r) = \log {\det}_2 (1+\Lambda),
\end{equation}
where
\begin{equation}\label{eq:deflambda}
    \Lambda = -\left(\begin{matrix}0 & R_1^{\frac12}qR_2^{\frac12} & R_1^{\frac12}rR_3^{\frac12}\\R_2^{\frac12}rR_1^{\frac12} & 0 & R_2^{\frac12}qR_3^{\frac12}\\ R_3^{\frac12}qR_1^{\frac12} & R_3^{\frac12}rR_2^{\frac12} & 0  \end{matrix}\right).
\end{equation}
Note that when $q,r \in L^2$,
\[
    A(k,q,r) = \log {\det}_2 (1+R_0(L-L_0)),
\]
and moreover, by Jacobi's formula
\begin{equation}\label{eq:logdettrace}
    A(k,q,r) = \sum_{l=2}^\infty \frac{(-1)^{l+1}}{l}\tr(\Lambda^l)
\end{equation}
With these definitions we prove:

\begin{lemma}
    Given $k > 0$, $q,r \in \Sc(\R)$, for all $1 \leq i,j \leq 3$ the operators $R_i^{\frac12}q R_j^{\frac12}$ are Hilbert-Schmidt on $L^2(\R)$ and satisfy
    \begin{equation}\label{eq:HSnorm}
        \|R_i^{\frac12}q R_j^{\frac12}\|_{\I_2}^2 \approx \int_\R \log\Big(4+ \frac{\xi^2}{k^2}\Big)\frac{|\hat q(\xi)|^2}{(\xi^2 + k^2)^{\frac12}}\, d\xi.
    \end{equation}
    Moreover, $\Lambda$ is Hilbert-Schmidt on $L^2(\R)^3$ with
    \begin{equation}\label{eq:HSnorm2}
        \|\Lambda\|_{\I_2}^2 \approx \int_\R \log\Big(4+ \frac{\xi^2}{k^2}\Big)\frac{|\hat q(\xi)|^2 + |\hat r(\xi)|^2}{(\xi^2 + k^2)^{\frac12}}\, d\xi.
    \end{equation}
\end{lemma}
\begin{proof}
    We begin by proving \eqref{eq:HSnorm}. For simplicity we consider $i = 1, j =2$, the other cases are proven similarly. We use Plancherel and the fact that $\omega^2 - \omega = \sqrt3 i$ to write
    \begin{align*}
        \|R_1^{\frac12}q R_2^{\frac12}\|_{\I_2}^2 &= \tr((R_1^*R_1)^{\frac12}q (R_2R_2^*)^{\frac12}\bar q)\\
        &= \tr((-\partial^2 + k^2)^{\frac12}q (-\partial^2 -\sqrt3 i k \partial + k^2)^{\frac12}\bar q)\\
        &= \int_\R \int_\R \frac{|\hat q(\xi - \eta)|^2}{(\xi^2 + k^2)^\frac12(\eta^2 + k^2 + \sqrt 3 k\eta)^\frac12}\, d\xi d\eta.
    \end{align*}
    The assertion now follows from the fact that
    \[
        \sqrt 3 k\eta \leq \frac{\sqrt 3}2(k^2 + \eta^2), \qquad \frac{\sqrt 3} 2 < 1,
    \]
    and
    \[
        \int \frac{dx}{((x+y)^2 + 1)^\frac12((x-y)^2 + 1)^\frac12} \approx \frac{\log(4 + 4y^2)}{(y^2 + 1)^\frac12},
    \]
    see also \cite[Lemma 4.1]{kvz}.
    
    \eqref{eq:HSnorm2} now follows from \eqref{eq:HSnorm} because
    \[
        \tr(\Lambda^* \Lambda) = \sum_{\sigma \in A_3}  \|R_{\sigma(2)}^{\frac12}q R_{\sigma(1)}^{\frac12}\|_{\I_2}^2 + \|R_{\sigma(1)}^{\frac12}r R_{\sigma(2)}^{\frac12}\|_{\I_2}^2,
    \]
    where $A_3$ is the alternating group of order $3$ (that is, $q$ is associated with the tuples $21, 32, 13$ and $r$ with $12, 23, 31$).
\end{proof}

As an important consequence of \eqref{eq:HSnorm2} we see that
\begin{equation}\label{eq:microlocalestimate}
    \|\Lambda\|_{\I_2}\lesssim k^{-\frac12-s}\|(q,r)\|_{H^s_k}, \quad s > -\frac12,
\end{equation}
and hence the series in \eqref{eq:logdettrace} converges geometrically for all $q,r \in H^s, s > -\frac12$ under the condition
\begin{equation}\label{eq:smallness}
    \|(q,r)\|_{H^s} \ll k^{\frac12+s}.
\end{equation}
For small data, the main contribution in \eqref{eq:logdettrace} comes from the quadratic term. The coercivity of the quadratic part in the case $r = \bar q$ was the key to prove low regularity a priori estimates for KdV, NLS and dNLS, and we have a coercivity here as well. More precisely:

\begin{lemma}
    The quadratic term in \eqref{eq:logdettrace} is
    \begin{equation}\label{eq:quadratic}
        \tr(\Lambda^2) = -3k \langle (-\partial^2 +3k^2 - \sqrt{3}ik \partial)^{-1}q, \bar r\rangle
    \end{equation}
    In the case $r = \bar q$ this quantity is coercive in the sense that
    \begin{equation}\label{eq:quadratic1}
        \tr(\Lambda^2) = \int_\R \frac{-3k|\hat q(\xi)|^2}{\xi^2 + 3k^2 - \sqrt{3}k\xi} \, d\xi \approx -\int_\R \frac{k|\hat q(\xi)|^2}{\xi^2 + 3k^2} \, d\xi.
    \end{equation}
\end{lemma}
\begin{proof}
    First note that by an explicit calculation,
    \[  
        \tr(\Lambda^2) = \tr(\tilde \Lambda^2) = 2\tr(R_1q R_2 r + R_2 q R_3 r + R_3 q R_1 r).
    \]
    When assuming $k > 0$ the operators $R_i$ have integral kernels 
    \begin{align*}
        K_1(x,y) &= - \chi_{\{x < y\}} e^{k(x-y)},\\
        K_2(x,y) &= \chi_{\{x > y\}} e^{\omega k(x-y)},\\
        K_3(x,y) &= \chi_{\{x > y\}} e^{\omega^2 k(x-y)},
    \end{align*}
    where $\chi_A$ denotes the characteristic function on the set $A$. Hence,
    \begin{align*}
        \tr(R_1 q R_2 r) &= \int_{\R^2} K_1(x,y) q(y) K_2(y,x) r(x)\, dydx \\
            &=  - \int_{x < y} e^{(1-\omega)k(x-y)}q(y) r(x) \, dydx\\
            &= \langle (\partial - (1-\omega)k)^{-1}q, \bar r\rangle,
    \end{align*}
    where $\langle \cdot, \cdot \rangle$ denotes the $L^2$ scalar product for complex valued functions. In a similar fashion,
    \begin{align*}
        \tr(R_3 q R_1 r) &= \int_{\R^2} K_3(x,y) q(y) K_1(y,x) r(x)\, dydx \\
            &=  - \int_{x > y} e^{(\omega^2-1)k(x-y)}q(y) r(x) \, dydx\\
            &= -\langle (\partial - (\omega^2-1)k)^{-1}q, \bar r\rangle,
    \end{align*}
    and
    \[
        \tr(R_2 q R_3 \bar r) = 0,
    \]
    because the supports of $K_2(x,y)$ and $K_3(y,x)$ do not overlap. Thus,
    \[
        \tr(\Lambda^2) = -3k \langle (-\partial^2 +3k^2 - \sqrt{3}ik \partial)^{-1}q, \bar r\rangle
    \]
    where we used $1 + \omega + \omega^2 = 0$ and $\omega - \omega^2 = \sqrt{3}i$.
    Thus by Plancharel, if $r = \bar q$,
    \[
        \tr(\Lambda^2) = \int_\R \frac{-3k}{\xi^2 + 3k^2 + \sqrt{3}k\xi}|\hat q(\xi)|^2 \, d\xi.
    \]
    For the last part we simply use the estimate
    \[
        \sqrt{3}k\xi \leq \frac12 \xi^2 + \frac32 k^2.
    \]
\end{proof}

Although not immediately needed, we present here the form of the terms of higher homogeneity. They can be used to obtain the asymptotic expansion of $A$ and thus the higher order energies for \eqref{eq:qdNLS}. We skip the proof as it is a simple calculation.
\begin{lemma}
    Then homogeneous terms of order three and four in \eqref{eq:logdettrace} satisfy
    \begin{align*}\label{eq:higherorder}
        \tr(\Lambda^3) &= - 3\tr(R_1 q R_2 q R_3 q + R_1 r R_2 r R_3 r),\\
        \tr(\Lambda^4) &= \tr\big((R_1 q R_2 r + R_1 r R_3 q)^2 + (R_2 q R_3 r + R_2 r R_1 q)^2 + (R_3 q R_1 r + R_3 r R_2 q)^2 \big)\\
                        &\quad +2\tr(R_1 q R_2 q R_3 r R_2 r + R_2 q R_3 q R_1 r R_3 r + R_3 q R_1 q R_2 r R_1 r).
    \end{align*}
\end{lemma}

The Green's function associated to the Lax operator $L$ is defined as the (by the Schwartz kernel Theorem) unique function $G(x,y)$ which is the integral kernel of $R = L^{-1}$, i.e.
\[
    L_x \int_\R G(x,y)f(y) \, dy = f(x),
\]
for a.e. $x \in \R$. The existence of such a function $G$ is ensured by the fact that $R$ is a Hilbert-Schmidt operator. Indeed, we simply write (compare to \cite[Section 3]{hgkv})
\[\begin{split}
    R -R_0 &= (L_0 + L - L_0)^{-1} - R_0  = R_0(1 + (L-L_0)R_0)^{-1} - R_0 \\
    &= \sqrt{R_0}\sum_{l=1}^\infty (-1)^l \big(\sqrt{R_0}(L-L_0)\sqrt{R_0}\big)^l \sqrt{R_0},
\end{split}\]
which is an absolutely convergent sum in the space of Hilbert Schmidt operators by \eqref{eq:HSnorm2} under the smallness condition \eqref{eq:smallness}, and can easily be checked by hand to be an inverse of $L$. In particular we can define the Green's function for $q,r \in H^s, s > -\frac12$ under the smallness condition \eqref{eq:smallness}.

The Green's function and its diagonal played a central role in the works for KdV, NLS and dNLS. Later in Lemma \ref{lem:logdetderivative} we will see that it also enters in the setting of qdNLS \eqref{eq:qdNLS} in a similar way: the functional derivatives of the logarithm of the transmission coefficient are given by sums of the off-diagonal matrix entries of the diagonal of the Green's function (i.e. its restriction to $x = y$). This relation will also allow us to prove that the inverse transmission coefficient coincides with the renormalized Fredholm determinant.

Our first result on the Green's function shows its close connection with the Jost solutions.

\begin{lemma}\label{lem:formofgreensfunction}
    Let $q,r \in L^2(\R)$. The constant multiple of the Green's function, $-T\cdot G(x,y)$, can be written as
    \begin{equation*}
        \left(\begin{smallmatrix}
            \Phi_{11}^-(x)(\Phi_{22}^+\Phi_{33}^+ - \Phi_{23}^+\Phi_{32}^+)(y) & \Phi_{11}^-(x)(\Phi_{32}^+\Phi_{13}^+ - \Phi_{33}^+\Phi_{12}^+)(y) & \Phi_{11}^-(x)(\Phi_{12}^+\Phi_{23}^+ -\Phi_{13}^+\Phi_{22}^+)(y)\\
            \Phi_{21}^-(x)(\Phi_{22}^+\Phi_{33}^+ - \Phi_{23}^+\Phi_{32}^+)(y) & \Phi_{21}^-(x)(\Phi_{32}^+\Phi_{13}^+ - \Phi_{33}^+\Phi_{12}^+)(y) & \Phi_{21}^-(x)(\Phi_{12}^+\Phi_{23}^+ -\Phi_{13}^+\Phi_{22}^+)(y)\\
            \Phi_{31}^-(x)(\Phi_{22}^+\Phi_{33}^+ - \Phi_{23}^+\Phi_{32}^+)(y) & \Phi_{31}^-(x)(\Phi_{32}^+\Phi_{13}^+ - \Phi_{33}^+\Phi_{12}^+)(y) & \Phi_{31}^-(x)(\Phi_{12}^+\Phi_{23}^+ -\Phi_{13}^+\Phi_{22}^+)(y)\end{smallmatrix}\right)
    \end{equation*}
    if $x < y$, and 
    \begin{equation*}
        \begin{split}
        \left(\begin{smallmatrix}
            \Phi_{12}^+(x)(\Phi_{21}^-\Phi_{33}^+ -\Phi_{31}^-\Phi_{23}^+)(y) & \Phi_{12}^+(x)(\Phi_{31}^-\Phi_{13}^+ -\Phi_{11}^-\Phi_{33}^+)(y) & \Phi_{12}^+(x)(\Phi_{11}^-\Phi_{23}^+ -\Phi_{21}^-\Phi_{13}^+)(y)\\
            \Phi_{22}^+(x)(\Phi_{21}^-\Phi_{33}^+ -\Phi_{31}^-\Phi_{23}^+)(y) & \Phi_{22}^+(x)(\Phi_{31}^-\Phi_{13}^+ -\Phi_{11}^-\Phi_{33}^+)(y) & \Phi_{22}^+(x)(\Phi_{11}^-\Phi_{23}^+ -\Phi_{21}^-\Phi_{13}^+)(y)\\
            \Phi_{32}^+(x)(\Phi_{21}^-\Phi_{33}^+ -\Phi_{31}^-\Phi_{23}^+)(y) & \Phi_{32}^+(x)(\Phi_{31}^-\Phi_{13}^+ -\Phi_{11}^-\Phi_{33}^+)(y) & \Phi_{32}^+(x)(\Phi_{11}^-\Phi_{23}^+ -\Phi_{21}^-\Phi_{13}^+)(y)
            \end{smallmatrix}\right)\\
            +
        \left(\begin{smallmatrix}
            \Phi_{13}^+(x)(\Phi_{31}^-\Phi_{22}^+ - \Phi_{21}^-\Phi_{32}^+)(y) & \Phi_{13}^+(x)(\Phi_{11}^-\Phi_{32}^+ - \Phi_{31}^-\Phi_{12}^+)(y) & \Phi_{13}^+(x)(\Phi_{21}^-\Phi_{12}^+ - \Phi_{11}^-\Phi_{22}^+)(y)\\
            \Phi_{23}^+(x)(\Phi_{31}^-\Phi_{22}^+ - \Phi_{21}^-\Phi_{32}^+)(y) & \Phi_{23}^+(x)(\Phi_{11}^-\Phi_{32}^+ - \Phi_{31}^-\Phi_{12}^+)(y) & \Phi_{23}^+(x)(\Phi_{21}^-\Phi_{12}^+ - \Phi_{11}^-\Phi_{22}^+)(y)\\
            \Phi_{33}^+(x)(\Phi_{31}^-\Phi_{22}^+ - \Phi_{21}^-\Phi_{32}^+)(y) & \Phi_{33}^+(x)(\Phi_{11}^-\Phi_{32}^+ - \Phi_{31}^-\Phi_{12}^+)(y) & \Phi_{33}^+(x)(\Phi_{21}^-\Phi_{12}^+ - \Phi_{11}^-\Phi_{22}^+)(y)
            \end{smallmatrix}\right),
        \end{split}
    \end{equation*}
    if $x > y$.
\end{lemma}
\begin{proof}
    By an abuse of notation we call $G(x,y)$ the matrix defined in the Lemma. Since the columns of $G(x,y)$ are multiples of Jost solutions they individually lie in the kernel of $L$. Hence $L_x G(x,y) = 0$ for all $x \neq y$. Moreover $G(\cdot,y) \in L^2(\R)$ because the left Jost solution $\Phi_1^-$ is integrable on the left of $y$ and the two right Jost solutions $\Phi_2^+, \Phi_3^+$ are integrable on the right of $y$. To show $L_x G(x,y) = \delta(x-y)$ it is enough to check that the jump condition
    \begin{equation}\label{eq:jump}
        G(y^+,y) - G(y^-,y) = \operatorname{Id}
    \end{equation}
    holds, where $y^\pm$ denotes the upper and lower limit towards $y$. Indeed, by writing $L = \partial - U(k,q,r)$, we see that for all $f \in C_c^\infty(\R)$ and $\varepsilon > 0$,
    \begin{align*}
        \int_{\R} L_x G(x,y)f(x) \, dx &= \int_{B_\varepsilon(y)} ((\partial_x - U(x)) G(x,y))f(x)\,dx\\
        &= -\int_{B_\varepsilon(y)}U(x)G(x,y)f(x) + G(x,y)\partial_x f(x)\,dx \\
        &\qquad + G(y+\varepsilon,y)f(y+\varepsilon) - G(y-\varepsilon,y)f(y-\varepsilon).
    \end{align*}
    Now since $G(\cdot,y),q,r \in L^2$ the integral on the right-hand side vanishes as $\varepsilon \to 0$, whereas the remaining part converges to
    \[
        (G(y^+,y)-G(y^-,y))f(y),
    \]
    which by assumption should be $f(y)$.
    
    Now \eqref{eq:jump} holds because for all off-diagonal entries the diagonal $x = y$ vanishes, whereas for the diagonal entries we use that
    \begin{align*}
        T^{-1} &= W(\Phi_{1}^-|\Phi_{2}^+|\Phi_{3}^+) \\
        &= \Phi_{11}^-(\Phi_{22}^+\Phi_{33}^+ - \Phi_{23}^+\Phi_{32}^+) - \Phi_{21}^-(\Phi_{12}^+\Phi_{33}^+ - \Phi_{13}^+\Phi_{32}^+) + \Phi_{31}^-(\Phi_{12}^+\Phi_{23}^+ - \Phi_{13}^+\Phi_{22}^+).
    \end{align*}
    is constant.
\end{proof}

When looking at Lemma \ref{lem:formofgreensfunction} one may wonder whether there is some more structure to the form of the Green's function. There is indeed, as is shown later in Lemma \ref{lem:formofgreensfunctiongeneral}.

We are immediately able to characterize the off-diagonal entries of the diagonal Green's function:

\begin{corollary}
    The off-diagonal entries of the diagonal Green's function are
    \begin{align}
        g_{21} &= -T\Phi_{21}^-(\Phi_{22}^+\Phi_{33}^+ - \Phi_{23}^+\Phi_{32}^+)\\
        g_{32} &= -T\Phi_{31}^-(\Phi_{32}^+\Phi_{13}^+ - \Phi_{33}^+\Phi_{12}^+)\\
        g_{13} &= -T\Phi_{11}^-(\Phi_{12}^+\Phi_{23}^+ -\Phi_{13}^+\Phi_{22}^+)\\
        g_{12} &= -T\Phi_{11}^-(\Phi_{32}^+\Phi_{13}^+ - \Phi_{33}^+\Phi_{12}^+)\\
        g_{23} &= -T\Phi_{21}^-(\Phi_{12}^+\Phi_{23}^+ -\Phi_{13}^+\Phi_{22}^+)\\
        g_{31} &= -T\Phi_{31}^-(\Phi_{22}^+\Phi_{33}^+ - \Phi_{23}^+\Phi_{32}^+)
    \end{align}
\end{corollary}

Next we show that the off-diagonal entries of the diagonal Green's function also enter as the functional derivative of the (renormalized) Fredholm determinant:

\begin{lemma}\label{lem:logdetderivative}
    We have
    \begin{align}
        \frac{\delta}{\delta q} \log {\det}_2(1+R_0(L-L_0)) &= - (g_{21}+g_{32}+g_{13}), \\
        \frac{\delta}{\delta r} \log {\det}_2(1+R_0(L-L_0)) &= - (g_{12}+g_{23}+g_{31}).
    \end{align}
\end{lemma}

This lemma follows from its generalization in Lemma \ref{lem:logdetderivativegeneral}. We postpone its proof until then because there is no further insight in its proof in the special case of \eqref{eq:qdNLS}. 

Finally, Lemma \ref{lem:logdetderivative} and Lemma \ref{lem:functionalderivative} allow to prove that inverse transmission coefficient and renormalized transmission coefficient coincide. In particular by Theorem \ref{thm:conservation} the renormalized transmission coefficient is a conserved quantity.

\begin{theorem}\label{thm:equality}
    For $q,r \in L^2(\R)$ and under the smallness condition \eqref{eq:smallness}, we have
    \begin{equation}\label{eq:equalityqdNLS}
        \log {\det}_2(1+R_0(L-L_0)) = -\log T.
    \end{equation}
\end{theorem}
\begin{proof}
    By Lemma \ref{lem:logdetderivative},
    \begin{align*}
        &\frac{\delta}{\delta q} \log {\det}_2(1+R_0(L-L_0)) = - (g_{21}+g_{32}+g_{13})\\
        &\quad = T\big(\Phi_{21}^-(\Phi_{22}^+\Phi_{33}^+ - \Phi_{23}^+\Phi_{32}^+) + \Phi_{31}^-(\Phi_{32}^+\Phi_{13}^+ - \Phi_{33}^+\Phi_{12}^+) + \Phi_{11}^-(\Phi_{12}^+\Phi_{23}^+ -\Phi_{13}^+\Phi_{22}^+)\big),
    \end{align*}
    and a similar statement holds for the functional derivative in $r$. But by Lemma \ref{lem:functionalderivative} this coincides with
    \[
        -\frac{\delta}{\delta q}\log T = T\frac{\delta}{\delta q} T^{-1}.
    \]
    Since the functional derivatives coincide we are done up to a constant. The statement thus follows using the fact that both sides in \eqref{eq:equalityqdNLS} vanish for $q = r = 0$.
\end{proof}

\section{Low regularity a priori estimates}\label{sec:apriori}

We now turn to the case $r = \bar q$ which corresponds to \eqref{eq:qdNLS}. After having established the conservation of
\[
    A(k,q) = \log {\det}_2(1+\Lambda)
\]
and due to identifying its bilinear term as the coercive quantity 
\[
    -\tr(\Lambda^2) = \int_{\R}\frac{3k|\hat q(\xi)|^2}{\xi^2 + 3k^2 + \sqrt{3}k\xi}\, d\xi,
\]
we can employ the machinery developed in \cite{kvz} to derive a priori estimates. In fact, we will restrict our attention to Sobolev norms in negative regularity, for which a priori estimates can be established by a simple integration argument. It is clear though that the results carry over to Besov norms as discussed in \cite{kvz}, and also to modulation space norms \cite{ohwang,klaus2} without any technical difficulties.

We define the microlocal Sobolev norms (which give the usual Sobolev norms when $k = 3^{-1/2}$)
\begin{equation}
    \|q\|_{H^{s}_k} = \Big(\int_{\R}|\hat q(\xi)|^2(\xi^2 + 3k^2)^{s}\, d\xi\Big)^{\frac12}.
\end{equation}

\begin{theorem}\label{thm:main1}
    Given a Schwartz solution $q$ of \eqref{eq:qdNLS}, its microlocal Sobolev norms are almost conserved for all $-\frac12 < s < 0$ and for $k \geq 1$ big enough depending on $\|q(0)\|_{H^{s}}$, in the sense that
    \begin{equation}\label{eq:unscaledconservation}
        \|q(t)\|_{H^s_k} \leq c\|q(0)\|_{H^s_k},
    \end{equation}
    for some constant $C > 0$. Moreover, for all $-\frac12 < s < 0$,
    \begin{equation}\label{eq:scaledconservation}
        \|q(t)\|_{H^s} \leq c(1+\|q(0)\|_{H^s})^{\frac{-s}{1+2s}}\|q(0)\|_{H^s}.
    \end{equation}
\end{theorem}
\begin{proof}
    We fix $k \geq 1$ big enough so that $k^{-\frac12-s}\|q(0)\|_{H^{s}_k} < \delta$ for $\delta$ small enough. Since $s > -\frac12$ this is true if
    \begin{equation}\label{eq:smallnessforseries}
        k^{s+\frac12} > \delta^{-1}\|q(0)\|_{H^{s}}.
    \end{equation}
    By making $\delta < 1$ a little smaller and since we are dealing with a Schwartz solution $q$ we can ensure that this condition also holds for a small open time interval $I$ around zero. The main observation now is that for all $s < 0$,
    \begin{equation}\label{eq:scoercive}
        \int_{k_0}^\infty \frac{3k^{1+2s}}{\xi^2 + 3k^2 +\sqrt{3}k\xi} \, dk \approx \int_{k_0}^\infty \frac{k^{1+2s}}{\xi^2 + k^2} \, dk \approx (\xi^2 + k_0^2)^s,
    \end{equation}
    as was already seen in \cite{kvz}. Thus we are led to define
    \[
        E_s(k_0,q) = \int_{k_0}^\infty k^{2s}A(k,q)\, dk,
    \]
    which is a conserved quantity because $A(k,q)$ is. Now since for $l \geq 2$ and $-\frac12 < s$, by \eqref{eq:microlocalestimate},
    \begin{align*}
        \int_{k_0}^\infty k^{2s}\big|\tr(\Lambda^l)\big|\, dk &\leq \int_{k_0}^\infty k^{-\frac l2-(l-2)s}\|q\|_{H^s_k}^l\, dk \leq \sup_{k \in [k_0,\infty)} \|q\|^l_{H^s_{k}}\int_{k_0}^\infty k^{-\frac l2-(l-2)s} dk\\
        &\leq k_0^{-\frac{l-2}2-(l-2)s}\|q\|_{H^s_{k_0}}^l,
    \end{align*}
    we obtain absolute convergence for the series expansion of $E_s(k,q) = \sum_{l=2}^\infty E_s^{(l)}(k,q)$ defined by
    \[
        E_s(k,q) = \sum_{l=2}^\infty \frac{(-1)^{l+1}}l \int_{k_0}^\infty \tr(\Lambda^l)\, dk.
    \]
    Moreover we find smallness for the tail of the series by our smallness assumption on $q$,
    \begin{equation}\label{eq:tailsmallness}
        \big|E_s(k,q) - E_s^{(2)}(k,q)\big| \leq \sum_{l=3}^\infty k^{-\frac{l-2}2-(l-2)s}\|q\|_{H^s_{k}}^l \leq c\delta E_s^{(2)}(k,q),
    \end{equation}
    where we notice that \eqref{eq:scoercive} implies
    \[
        E_s^{(2)}(k,q) = -\frac12\int_{k_0}^\infty \tr(\Lambda^2)\, dk \approx \|q\|_{H^s_k}^2.
    \]
    Hence $E_s(k,q)$ is a conserved quantity which is coercive against the $H^s_k$ norm,
    \[
        E_s(k,q) \geq (1-c\delta)E_s^{(2)}(k,q) \geq C(1-c\delta)\|q\|_{H^s_k}^2,
    \]
    thus proving for small times,
    \begin{equation}\label{eq:conservationsmall}
        \|q(t)\|_{H^s_k}^2 \leq \frac{C(1+\delta)}{1-c\delta}\|q(0)\|_{H^s_k}^2.
    \end{equation}
    But the same estimate also holds for large times provided we can make sure that $\|q(t)\|_{H^s_k}^2 \leq k^{\frac12+s}\delta$, which we can by \eqref{eq:conservationsmall} and by choosing $\delta$ a little smaller.

    Finally \eqref{eq:scaledconservation} follows from \eqref{eq:unscaledconservation} by combining (for $s < 0$)
    \[
        \|q\|_{H^s}^2 \lesssim (1 + k^{-2s})\|q\|_{H^s_k}^2
    \]
    with \eqref{eq:smallnessforseries}.
\end{proof}

\section{Fredholm determinant for N x N Lax operators}\label{sec:equality2}

In this section we want to generalize Theorem \ref{thm:equality} to $N \times N$ Lax pairs. We consider the first order ODE problem
\begin{equation}\label{eq:generalfirstorder}
    \phi_x = (kJ + U_0(u))\phi,
\end{equation}
with the following assumptions:
\begin{enumerate}
    \item $J$ is diagonal, $J = \operatorname{diag}(\omega_1,\dots,\omega_n)$, with $n$ distinct eigenvalues, $\omega_i \neq \omega_j$ if $i \neq j$, and $J$ is trace free, $\tr(J) = 0$.
    \item $U_0(u)$ is off-diagonal, $U_0(u)_{ii} = 0$ for $1 \leq i \leq n$. Its entries are polynomials without zeroth order term in the components of $u = (u_1,\dots,u_m)$.
    \item $u_i \in C_c^\infty(\R)$ for all $i$.
\end{enumerate}
In particular $U_0(u) \in C_c^\infty(\R,\C^{n\times n})$. Again we define $K = \overline{\cup_i \supp u_i}$.

Our first mission then is to define what we mean by Jost solutions in this context. To simplify the analysis we assume that $k > 0$ is real-valued and positive and that the entries of $J$ are ordered.
\begin{equation}\label{eq:ordering}
    \real(\omega_1) \geq \dots \geq \real(\omega_l) > 0 > \real(\omega_{l+1}) \geq \dots \geq \real(\omega_{n}).
\end{equation}
The assumption that $\real (\omega_i) \neq 0$ is necessary to have decay either as $x \to \infty$ or as $x \to -\infty$.

With these assumptions in place we define the Jost solutions $\Phi_j^{\pm}$ similar to Section \ref{sec:transmission}:

\begin{definition}\label{def:jostsolution}
    Given solutions $\Phi_j^{\pm}$ of \eqref{eq:generalfirstorder} such that
    \begin{equation}\label{eq:cond1}
        e^{\omega_j k x} \Phi_j^{\pm}(x) = e_j, \qquad x \to \pm \infty,
    \end{equation}
    we call $\Phi^-_j$ a left Jost solution if $\real (\omega_j) > 0$ and $\Phi^+_j$ a right Jost solution if $\real (\omega_j) < 0$.
\end{definition}
With other words, left Jost solutions are decaying unit vector exponentials at $-\infty$ and right Jost solutions are are decaying unit vector exponentials at $\infty$.

Again the existence of Jost solutions is ensured by the Cauchy-Lipschitz theorem. By the ordering of eigenvalues \eqref{eq:ordering}, there are exactly $l$ left Jost solutions and $n-l$ right Jost solutions,
\[
    \Phi_1^-, \dots, \Phi_l^-, \quad \text{and}\quad \Phi_{l+1}^+, \dots, \Phi_n^+.
\]
The transmission coefficient corresponding to $L = \partial_x - kJ - U_0(u)$ is defined as
\begin{equation}
    T^{-1}(k,u) := \det(\Phi_1^-| \dots|\Phi_l^-|\Phi_{l+1}^+| \dots| \Phi_{n}^+).
\end{equation}

A scattering theory for the problem \eqref{eq:generalfirstorder} for $u_i \in \Sc(\R)$ (respectively $u_i \in L^1(\R)$) was developed in the work of Beals--Coifman \cite{bealscoifman}. As was mentioned in Section \ref{sec:transmission} and also outlined in \cite{bealscoifman,deift} if $u$ is not compactly supported, in addition to $e^{\omega_j k x} \Phi_j^{\pm}(x) \to e_j, x \to \pm \infty$ one has to impose the extra condition
\begin{equation}\label{eq:cond2}
    e^{\omega_j k x} \Phi_j^{\pm}(x) \in L^\infty(\R,\C^n),
\end{equation}
to obtain an unambiguous (and different!) definition. Indeed, if e.g. $\real (\omega_1) > \real (\omega_2$), then if $\psi_2$ satisfies \eqref{eq:cond1} so does $\psi_2 + \psi_1$. Note that this implies that the left Jost solutions in \cite{bealscoifman} are possibly linear combinations of the Jost solutions in here, and similar for the right Jost solutions.

We do not need this full theory here because we restrict to the analysis of the (inverse) transmission coefficient. Its definition is independent of whether one takes Definition \ref{def:jostsolution} or the one from \cite{bealscoifman}. Indeed, suppose that $\real(\omega_1) > \dots > \real(\omega_l) > 0 > \real(\omega_{l+1}) > \dots > \real(\omega_{n})$. Then if $\Phi_i^{\pm,BC}$ denote the Jost solutions defined with the additional condition \eqref{eq:cond2}, we find that
\begin{align*}
    \Phi_1^{-,BC} &= \Phi_1^{-}, \quad \Phi_2^{-,BC} = \Phi_2^{-} + a_{21}\Phi_1^{-}, \quad \Phi_3^{-,BC} = \Phi_3^{-} + a_{31}\Phi_1^{-} + a_{32}\Phi_2^{-}, \dots\\
    \Phi_n^{+,BC} &= \Phi_n^{+}, \quad \Phi_{n-1}^{+,BC} = \Phi_{n-1}^{+} + a_{n-1 n}\Phi_n^{-}, \dots.
\end{align*}
But then,
\[
    T^{-1} = W(\Phi_1^-|\dots|\Phi_n^+) = W(\Phi_1^{-,BC}|\dots|\Phi_n^{+,BC}),
\]
because all other combinations vanish in the Wronskian. Another advantage is that by using compactly supported functions we can assure the existence of Jost solutions for every $k > 0$ even without imposing a scaling invariant $L^1$ smallness condition as in \cite[Theorem 3.8]{bealscoifman}.

We want to prove Theorem \ref{thm:main2small} by showing that the functional derivatives of inverse transmission coefficient and renormalized Fredholm determinant coincide. To this end we have to analyze the Green's function of the operator $L$, more precisely its diagonal. 

We start with the Green's function itself. We write
\[
    R_0 = (\partial - kJ)^{-1}.
\]
Since $\real (\omega_j) \neq 0$, $R_0$ has an integral kernel $G_0$ given by a diagonal matrix with kernels as in \eqref{eq:integralkernel} on the diagonal. We write again
\begin{equation}\label{eq:absolutelyconvergentsum}
    R - R_0= \sqrt{R_0}\sum_{l=1}^\infty (-1)^l (\sqrt{R_0}U_0(u)\sqrt{R_0})^l \sqrt{R_0}.
\end{equation}
Since for all $i \neq j$, by a calculation similar to \eqref{eq:HSnorm},
\[
    \|(\partial - k \omega_i)^{-\frac12}(U_{0})_{ij}(u)(\partial - k \omega_j)^{-\frac12}\|_{\I_2} \approx \int_\R \log\Big(4+\frac{\xi^2}{k^2}\Big)\frac{|(\hat U_{0})_{ij}(u)(\xi)|^2}{(\xi^2 + k^2)^{\frac12}}\, d\xi,
\]
we find under the smallness assumption
\begin{equation}\label{eq:smallness3}
    \|U_0(u)\|_{H^s} \lesssim k^{\frac12+s}, \qquad s > -\frac12,
\end{equation}
that \eqref{eq:absolutelyconvergentsum} has an absolutely convergent right-hand side, and that $R-R_0$ is Hilbert-Schmidt.

Even more, by arguing as in \cite[Proposition 3.1]{hgkv} it can be seen that $R-R_0$ has an integral kernel which is continuous on the restriction to $x = y$. We call this function $\tilde g(x)$. This continuity can also be derived from the following lemma which is a generalization of Lemma \ref{lem:formofgreensfunction}. We do not necessarily need it for our later calculations, but it gives a nice background on how to construct the Green's function if one is given the Jost solutions.

Recall that the tensor product $v \otimes w$ of two vectors $v,w$ is the matrix defined by $(v \otimes w)_{ij} = v_i w_j$.

\begin{lemma}\label{lem:formofgreensfunctiongeneral}
    The Green's function $G(x,y)$ of the first order operator $L = \partial - kJ - U_0(u)$ is given by
    \begin{equation}
        G(x,y) = \begin{cases}
            -T(\Phi_1^-(x)\otimes v_1(y) + \dots +\Phi_l^-(x)\otimes v_l(y)), \qquad \text{if} \quad x < y,\\
            T(\Phi_{l+1}^+(x)\otimes v_{l+1}(y) + \dots +\Phi_n^-(x)\otimes v_n(y)), \qquad \text{if} \quad x > y,
        \end{cases}
    \end{equation}
    where the vectors $v_i(y)$ are defined by duality satisfying
    \begin{equation}
        \langle v_i(y),w\rangle_{\C^n} = \det(\Phi_1^-(y)|\dots|\Phi_{i-1}^{\pm}(y)|\,w\,|\Phi_{i+1}^{\pm}(y)|\dots|\Phi_n^+(y)), \qquad \forall\, w \in \C^n.
    \end{equation}
\end{lemma}
\begin{proof}
    Again we check easily that the columns of $G$ solve the equation $L_x G(x,y) = 0$ for all $x\neq y$ and are $L^2$ because they are exponentially decaying off $y$. Thus it remains to check the Jump condition \eqref{eq:jump}. To this end notice that by definition e.g.
    \begin{align*}
        &T^{-1}\langle e_1, (G(y^+,y) - G(y^-,y))e_1\rangle  \\
        &\qquad = \Phi_{11}^-\det(e_1|\Phi_2^-|\dots|\Phi_n^+) + \dots +\Phi_{1n}^+\det(\Phi_1^-|\dots|\Phi_{n-1}^+|e_1) = T^{-1},
    \end{align*}
    by expanding the determinant representation of $T^{-1}$ by the first row. In a similar manner we see that all other diagonal entries satisfy the jump condition.
\end{proof}

After all these preparations, we can calculate the functional derivative of the inverse transmission coefficient, giving Lemma \ref{lem:functionalderivative} the following generalization:

\begin{lemma}\label{lem:functionalderivativegeneral}
    The transmission coefficient corresponding to the operator $L = \partial - kJ -U_0(u)$ satisfies
    \begin{equation}\label{eq:functionalderivativegeneral}
        \frac{\delta}{\delta u_i} \log T^{-1} = \tr(\nabla_i U_0(u) \tilde g),
    \end{equation}
    where $\frac{d}{ds}|_{s=0} U_0(u+sve_i) = (\nabla_i U_0)(u)v$, $1 \leq i \leq m$.
\end{lemma}

\begin{proof}
    Write $\dot U = \frac{d}{ds}\Big|_{s=0} U_0(u+sve_j)$, and correspondingly for $\Phi$. Then,
    \[
        L\dot\Phi = \dot U \Phi
    \]
    for the Jost solutions. Since we assumed $u \in C^\infty_c$, the right-hand side is in $L^2(\R)$ and we can unambiguously write
    \[
        \dot \Phi(x) = \int_\R G(x,y)\dot U(y)\Phi(y)\, dy.
    \]
    This formula holds both when $\Phi$ is a single Jost solution, or a matrix made of Jost solutions, that is
    \[
        \frac{d}{ds}\Big|_{s=0}(\Phi_1^-|\dots|\Phi_n^+) = \int_\R G(x,y)\dot U(y)(\Phi_1^-|\dots|\Phi_n^+)(y)\, dy.
    \]
    Thus, for all $x \in \R$, by making use of Jacobi's formula for the derivative of the determinant,
    \begin{align*}
        \frac{d}{ds}\Big|_{s=0} &\log\det\Big[(\Phi_1^-|\dots|\Phi_n^+)(x)\Big] \\
        &= \tr\Big[(\Phi_1^-|\dots|\Phi_n^+)^{-1}(x)\int_\R G(x,y)\dot U(y)(\Phi_1^-|\dots|\Phi_n^+)(y)\, dy\Big]\\
        &= \int_\R \tr\Big[(\Phi_1^-|\dots|\Phi_n^+)^{-1}(x)G(x,y)\dot U(y)(\Phi_1^-|\dots|\Phi_n^+)(y)\Big]\, dy\\
        &= \int_\R \tr\Big[(\Phi_1^-|\dots|\Phi_n^+)(y)(\Phi_1^-|\dots|\Phi_n^+)^{-1}(x)G(x,y)\dot U(y)\Big]\, dy,
    \end{align*}
    where in the last line we cycled the trace. Thus for all $x \in \R$ and a.e. $y \in \R$,
    \[
        \frac{\delta}{\delta u_i} \log T^{-1}(y) = \tr\Big[(\Phi_1^-|\dots|\Phi_n^+)(y)(\Phi_1^-|\dots|\Phi_n^+)^{-1}(x)G(x,y)(\nabla_i U_0)(u)\Big].
    \]
    We write $G = G_0 + G - G_0$ and note that setting $x = y$ in the summand with $G-G_0$ gives exactly the right-hand side of \eqref{eq:functionalderivativegeneral}. It remains to show that the summand with $G_0$ vanishes when $x \to y, x \neq y$. Indeed, because $G_0(x,y)\nabla_i U$ is a trace-free matrix,
    \begin{equation}\begin{split}\label{eq:conv}
        \Big|\tr\big[\boldsymbol{\Phi}(y)\boldsymbol{\Phi}^{-1}(x)G_0(x,y)\nabla_i U_0\big]\Big| &= \Big|\tr\big[(\boldsymbol{\Phi}(y)-\boldsymbol{\Phi}(x))\boldsymbol{\Phi}^{-1}(x)G_0(x,y)\nabla_i U_0\big]\Big|\\
        &\leq |\boldsymbol{\Phi}(y)-\boldsymbol{\Phi}(x)|\boldsymbol{\Phi}^{-1}(x)||G_0(x,y)||\nabla_i U_0|
    \end{split}\end{equation}
    where we introduced $\boldsymbol{\Phi} = (\Phi_1^-|\dots|\Phi_n^+)$. Now since we may write $\boldsymbol{\Phi}^{-1}$ in terms of its cofactor matrix and from $\det {\Phi} = T^{-1}$, we can estimate
    \[
        |\boldsymbol{\Phi}^{-1}(x)| \lesssim T |\boldsymbol{\Phi}(x)|^{n-1},
    \]
    and because $|G_0(x,y)| \leq e^{|x-y|}$, both sides in \eqref{eq:conv} vanish as $x \to y$ by continuity of $\boldsymbol{\Phi}$.
\end{proof}

An important case of Lemma \ref{lem:functionalderivativegeneral} is when $U_0$ is linear in $u$ and has the form
\[
    U_0(u) = u_1 A_1 + \dots + u_m A_m.
\]
In this case the derivative reduces to
\[
    \frac{\delta}{\delta u_i} \log T^{-1} = \tr(A_i \tilde g),
\]
which may be compared to Lemma \ref{lem:logdetderivative}. Concerning Lemma \ref{lem:logdetderivative} we find the following generalization to hold:

\begin{lemma}\label{lem:logdetderivativegeneral}
    Under the smallness assumption \eqref{eq:smallness3} the renormalized Fredholm determinant satisfies
    \begin{equation}\label{eq:logdetderivativegeneral}
        -\frac{\delta}{\delta u_i} \log {\det}_2(1 -  R_0U_0(u)) = \tr(\nabla_i U_0(u) \tilde g).
    \end{equation}
\end{lemma}

\begin{proof} 
    Note that we can write $\tilde g(x) = ((R-R_0)\delta_x)(x)$. Thus we calculate with $\hat v = ve_i$, $e_i$ being the $i$th unit vector,
    \begin{align*}
        LHS\eqref{eq:logdetderivativegeneral} &= -\frac{d}{ds}\Big|_{s=0} \log {\det}_2(1 - R_0U_0(u+s\hat v)) \\
        &= \frac{d}{ds}\Big|_{s=0} \sum_{l=2}^\infty \frac{1}{l} \tr(R_0U_0(u+s\hat v))^l\\
        &= \sum_{l=2}^\infty \tr\big[(R_0U_0(u))^{l-1}R_0 \nabla_i U(u)v\big]\\
        &= \tr\Big[\sum_{l=2}^\infty \nabla_i U_0(u)v(R_0U_0(u))^{l-1}R_0\Big]\\
        &= \int \tr\Big[v(x)\nabla_i U_0(u(x))\Big(\sum_{l=2}^\infty (R_0U_0(u))^{l-1}R_0 \delta_x\Big)(x)\Big]\,dx\\
        &= \int v(x)\tr\big[\nabla_i U_0(u(x))((R-R_0)\delta_x)(x))\big]\,dx\\
        &= \int v(x)\tr[\nabla_i U_0(u(x)) \tilde g(x)]\, dx,
    \end{align*}
    where used the definition of the renormalized Fredholm determinant, calculated the derivative explicitly, cycled the trace, used the definition of the trace as the integral over the diagonal of the kernel, and observed the appearence of $R-R_0$ in the integral.
\end{proof}

By combining Lemma \ref{lem:functionalderivativegeneral} with Lemma \ref{lem:logdetderivativegeneral} as well as with the fact that
\[
    T^{-1}(k,0) = 1 = {\det}_2(1)
\]
we arrive at our second main Theorem:

\begin{theorem}\label{thm:main2}
    Under the assumption of the beginning of Section \ref{sec:equality2} and under the smallness assumption \eqref{eq:smallness3}, we have equality
    \begin{equation}
       T^{-1}(k,u) = {\det}_2\big(1-(\partial-kJ)^{-1}U_0(u)\big),
    \end{equation}
    of inverse transmission coefficient and renormalized Fredholm determinant.
\end{theorem}

\section{Open questions}\label{sec:questions}

As is usual in mathematics, answering one question rises at least ten new questions. Some of them are just not clear to the author and may have been already treated elsewhere, some may be of general interest. A few of them which were not answered in this work are listed here. If one of the readers has an answer to one of the questions, the author would be grateful if he or she let him know. 

\begin{enumerate}
    \item Is it also possible to treat $U_0(u)$ which has differential polynomials as off-diagonal entries in Theorem \ref{thm:main2}? Having $\nabla_i U_0(u)$ contain differential operators as entries causes problems in Lemma \ref{lem:functionalderivativegeneral} and Lemma \ref{lem:logdetderivativegeneral}. Having $U_0(u)$ being also dependent on $k$ seems to be off no issue, thus one could hope then to analyze the good Boussinesq equation using the $3\times3$ Lax pair from \cite[Section 3]{gb}
    \item Is it possible to give a general formula for the density function of $\log {\det}_2(1 - (\partial-kJ)^{-1}U_0(u))$ similar to \cite{hgkv}? If so, can it be used to show conservation?
    \item Is \eqref{eq:qdNLS} amenable to the method of commuting flows (see \cite{hgknv})?
    \item Is it possible to construct weak solutions in negative regularity to the good Boussinesq equation by using the Miura map of \cite{charlierlenells}, arguing similar to \cite{kkl}? The Miura map is defined as follows: Let
    \[
        u = - \frac92 |q|^2 + 3\real (\omega q_x).
    \]
    Then if $q$ solves
    \[
        iq_t - \frac1{\sqrt{3}}q_{xx} + 2\sqrt3\bar q \bar q_x = 0,
    \]
    $u$ gives rise to a solution of
    \[
        u_{tt} + \frac13u_{xxxx} + \frac43 (u^2)_{xx} = 0.
    \]
    \item Can one also invert the Miura map of \cite{charlierlenells} by modifying it similar to \cite{kkl}?
    \item Is this Miura map related to a factorization of the third order one-dimensional Lax operator of good Boussinesq similar to what happens for the Miura map relating KdV and mKdV resp. Gardner (see \cite{kkl})? This may be related to the Drinfeld-Sokolov reduction.
    \item The vector and matrix NLS equations are examples of equations with higher dimensional Lax pairs but seem to fail the condition (1) from the beginning of Section \ref{sec:equality2}. Can one still obtain results for these equations?
    \item We have now seen three different one-dimensional Schrödinger equations which are integrable: cubic NLS, cubic dNLS and the quadratic dNLS \eqref{eq:qdNLS}. If the author is not mistaken, it seems that given a dNLS with two derivatives on the cubic term, one can modify the proof of conservation of the determinant in \cite{klaus} to still hold. This would yield another integrable nonlinear Schrödinger equation. Maybe it is also possible to construct a hierarchy related to the one of \eqref{eq:qdNLS} in the same way as the Kaup-Newell hierarchy is related to the AKNS hierarchy. Is there a general theory to construct integrable nonlinear Schrödinger equations arising from $N \times N$ Lax pairs?
    \item Even further, is there maybe a way to construct one-dimensional integrable PDE as reductions of systems integrable by $N \times N$ pairs?
\end{enumerate}

\newpage

\bibliographystyle{plain}

\end{document}